\documentclass[11pt]{article}
\usepackage{ulem}
\usepackage{graphicx,amsmath,amsfonts,amssymb,color,bbm}
\usepackage{multirow}
\usepackage{epsfig}

 \newcommand{\pend}{\hfill \thicklines \framebox(5.5,5.5)[l]{}}
 \newenvironment{proof}{\noindent {\sc  Proof.} \rm}{\pend}

\numberwithin{equation}{section}

 \newtheorem{theorem}{Theorem}
 \newtheorem{lemma}{Lemma}[section]

 \newtheorem{remark}{Remark}[section]

 \newtheorem{corollary}{Corollary}[section]
 \newtheorem{definition}{Definition}[section]

\begin{document}
 \pagenumbering{arabic} \thispagestyle{empty}
\setcounter{page}{1}

\begin{center}
\begin{center}{\Large{\bf Second Order Asymptotic Properties for the Tail Probability of the Number of Customers
in the $M/G/1$ Retrial Queue}}

\

 { Bin Liu
%\footnotemark[1]$^{,a}$
$^{a}$  and Yiqiang Q. Zhao $^b$
 \\ {\small a. School of Mathematics and Physics, Anhui Jianzhu University,
Hefei 230601, P.R. China}\\
{\small b. School of Mathematics and Statistics,
Carleton University, Ottawa, ON, Canada K1S 5B6}}
\end{center}

December 25, 2018
\date{\today}

\end{center}
%\footnotetext[1]{Corresponding author: Bin Liu, E-mail address:
%bliu@amt.ac.cn}

\begin{abstract}
When an explicit expression for a probability distribution function $F(x)$ can not be found, asymptotic properties of the tail probability function $\bar{F}(x)=1-F(x)$ are very valuable, since they provide approximations for system performance and approaches for computing probability distribution. In this paper, we study tail asymptotic properties for the number of customers in the $M/G/1$ retrial queueing system with heavy-tailed service time. For queueing systems, studies on asymptotic tails are mainly concentrated on the first order asymptotic behaviour. To the best of our knowledge, there is no second order tail asymptotic analysis for retrial queueing models. Second order asymptotic expansions provide the refined asymptotic results on the first order approximation, and are often more difficult to obtain, as expected. The main contribution of this paper is the second order asymptotic analysis for the $M/G/1$ retrial queue, which provides the information about the convergence speed of the first order approximation.
\medskip

\medskip


\end{abstract}

\section{Introduction}

Stationary probability distributions, such as the stationary queue length distribution, for queueing systems are one of the most important types of performance measures, which often provide fundamental information about the system required in performance evaluation, system design, control and optimization. A large number of queueing systems do not possess explicit expressions for
such stationary distributions. When an explicit expression is not available, asymptotic properties of the tail probabilities often lead to approximations for system performance, and also computational approaches.

In the queueing literature, for a distribution function $F(x)$, tail asymptotic analysis has mainly focused on the so-called first order approximation for the tail probability (or distribution) function $\bar{F}(x)=1-F(x)$. Specifically, we want to identify a function $\bar{F}_1(x)$, referred to as the first (dominant) term in the asymptotic expansion, such that $\lim_{x \to \infty} \bar{F}(x)/\bar{F}_1(x)=1$,
denoted by $\bar{F}(x) \sim \bar{F}_1(x)$,
which is equivalent to the first order asymptotic expansion
$\bar{F}(x) = \bar{F}_1(x) + o(\bar{F}_1(x))$ as $x\to\infty$,
where $o(\bar{F}_1(x))$ is a higher order infinitesimal function compared to $\bar{F}_1(x)$.
The second order approximation is a refinement of the first order approximation. Specifically, besides the first order term $F_1(x)$, we want to further identify a second function $\bar{F}_2(x)$ satisfying $\bar{F}(x) - \bar{F}_1(x) \sim \bar{F}_2(x)$, or
%\begin{equation} %\label{eqn:2nd order}
     $\bar{F}(x) = \bar{F}_1(x) + \bar{F}_2(x) + o(\bar{F}_2(x))$ as $x\to\infty$.
%\end{equation}

In this paper, we are interested in the second order asymptotic expansion for the stationary queue length tail probability function for the $M/G/1$ retrial queueing system. The $M/G/1$ retrial queue considered in this paper is a standard single server retrial queueing model, in which the primary customers arrive according to a Poisson process with rate $\lambda$. An arrived (primary) customer, to the system with the server idle, is served immediately by the server, otherwise it joins the orbit of infinite waiting capacity, becoming a repeated customer. Each of the repeated customers in the orbit independently repeatedly tries (visits to the server) for receiving service until it finds the server idle and then the service starts immediately. All inter-visit times are i.i.d. exponential r.v.'s with (retrial) rate $\mu$. All service times (for primary and repeated customers) are assumed to be i.i.d. r.v.'s (also independent of arrivals) having the distribution $F_{\beta}(t)$ with $F_{\beta}(0)=0$ and a finite mean $\beta_1$. Upon completion of its service, the (primary or repeated) customer leaves the system. Let $T_\beta$ be the generic service time having distribution function $F_{\beta}(t)$, and $\beta(s)$ the Laplace-Stieltjes transforms (LST) of $F_{\beta}(t)$.
Let $\rho=\lambda\beta_1$, be the traffic intensity. It is well known that the system is stable if and only if (iff) $\rho<1$ (\cite{Falin-Templeton:1997}, p.20), which is assumed to hold throughout this paper.

Retrial queues are a  type of very important queueing systems, which find many applications in abroad range of areas and have been extensively studied for more than 40 years. Literature reviews on retrial queues can be found in books or recent surveys, for example,
Falin and Templeton(1997)~\cite{Falin-Templeton:1997},
Artajelo and G\'{o}mez-Corral(2008)~\cite{Artalejo-GomezCorral:2008},
Choi and Chang(1999)~\cite{Choi-Chang:1999},
Kim and Kim(2016)~\cite{Kim-Kim:2016}, and Phung-Duc(2017)~\cite{Phung-Duc:2017}.
Tail asymptotic properties for retrial queues have been reported in, for example,
Shang, Liu and Li (2006)~\cite{Shang-Liu-Li:2006}, Kim, Kim and Ko (2007)~\cite{Kim-Kim-Ko:2007}, Kim, Kim and Kim (2010)~\cite{Kim-Kim-Kim:2010a}, Kim, Kim and Kim (2010, 2012)~\cite{Kim-Kim-Kim:2010c, Kim-Kim-Kim:2012}, Liu and Zhao (2010)~\cite{Liu-Zhao:2010}, Kim and Kim (2012)~\cite{Kim-Kim:2012}, Liu, Wang and Zhao (2012)~\cite{Liu-Wang-Zhao:2012}, Yamamuro (2012)~\cite{Yamamuro:2012},
%Artalejo and Phung-Duc (2013)~\cite{Artalejo-PhungDuc:2013},
Liu, Wang and Zhao (2014)~\cite{Liu-Wang-Zhao:2014}, Masuyama (2014)~\cite{Masuyama:2014},
%Kim (2015)~\cite{Kim:2015},
and very recently Liu, Min and Zhao (2017)~\cite{Liu-Min-Zhao:2017}.

Most of the reported results on tail asymptotics for queueing systems are the first order approximations and the volume of research papers on second order approximations is much smaller.
Discussions on the second (or higher) order asymptotic expansions for a tail probability function for queueing systems have, insofar, mainly focused on the tail probabilities of the waiting time, and none of them are for retrial queues. For example, multi-term asymptotic expansions are obtained for the waiting time tail distribution for the standard (non-retrial) $M/G/1$ queues, such as:
in Willekens and Teugels (1992)~\cite{Willekens1992}, under the assumption of subexponential service time distributions;
in Abate, Choudhury and Whitt (1994)~\cite{Abate-Choudhury-Whitt:1994}, under the assumption of Pareto mixture of exponential (PME) service times (Section~3);
in Boxma and Cohen (1998)~\cite{Boxma-Cohen:1998}, for a class of special heavy-tailed service time distributions; and
in Abate and Whitt (1999) \cite{Abate-Whitt:1999}, an extension of the work in~\cite{Boxma-Cohen:1998}.

The second order asymptotic is of fundamental interest in applied probability and statistics. It has been proven that second and higher order properties of regular variation functions are very useful for studying the convergence rate of the extreme order statistics in extreme value theory (see de Haan and Resnick~\cite{DeHaan-Resnick:1996}), characterizing asymptotic normality of Hill's estimator (Geluk et al.~\cite{Geluk-DeHaan-Resnick-Starica:1997}), and establishing the second-order asymptotics of risk measures and their risk concentrations (Degen et al.~\cite{Degen-Lambrigger-Segers:2010}; Hua and Joe~\cite{Hua-Joe:2011}). However, only limited research (e.g., \cite{Willekens1992}) can be found for queueing models on this topic.
It is worthwhile to emphasize that the second order approximation to queueing systems is also of fundamental interest, because:
%\begin{description}
%\item[(1)]
(1) The second order property determines the rate of convergence for the first order approximation. When the speed of convergence is higher, the asymptotic approximation by the first order term has a better chance of providing a good approximation;
%\item[(2)]
(2) Heavy-tailed service times usually represent atypically or extremely long processing times in telecommunications networks, which result in huge waiting times and queue lengths in the system. The second order asymptotic property provides the information about the sub-extremal level of long processing times, which may not be negligible in the tail asymptotic approximation. For instance, consider the tail probability $\bar{F}_{\beta}(t)=0.01t^{-2.01}+100t^{-1.99}$;
%\item[(3)]
(3) The second asymptotic term itself improves the quality of approximations.
%\end{description}
Our result is a complement to the queueing literature, particularly on retrial queues.

Our focus in this paper is to derive the  second order approximation, or two-term asymptotic expansion, for the total number $L_\mu$ of customers in the $M/G/1$ retrial queueing system for a type of heavy-tailed service time distributions, which is specified by the following assumption:

\noindent{\bf Assumption A.}
{\it
$\bar{F}_{\beta}(t)=t^{-a}L(t)$ and $L(t)=r_1+r_2t^{-h}L_0(t)$ as $t\to\infty$, where $a>2$, $h>0$, $r_1>0$, $-\infty<r_2<\infty$ and $L_0(t)$ is a slowly varying function at infinity.
}

Examples of probability distributions satisfying Assumption A include:
\begin{description}
\item[(i)] Hall/Weiss Class (\cite{Weiss:1971}, \cite{Hall:1982}):
$\bar{F}_{\beta}(t)=\displaystyle(1/2)t^{-v}\left(1+t^{-w}\right)$ with $t\ge 1$, where $v>2$ and $w>0$.

\item[(ii)] Burr distributions (including Lomax distributions):
\begin{eqnarray}
\bar{F}_{\beta}(t)&=&\displaystyle\left(\frac {b} {b+t^w}\right)^{v},\  \mbox{where } b,v,w>0 \  \mbox{and }vw>2,\nonumber
\end{eqnarray}
thus $\bar{F}_{\beta}(t)=b^v t^{-vw}\left[1-vb t^{-w}+O(t^{-2w})\right]$ as $t\to\infty$.

\item[(iii)] Folded student's $t$-distributions:
\begin{eqnarray}
\bar{F}_{\beta}(t)&=&2\cdot\displaystyle\frac {\Gamma((v+1)/2)}{\sqrt{v\pi}\Gamma(v/2)}\int_t^{\infty}\left(1+\frac {x^2}{v}\right)^{-{(v+1)}/2}dx,\  \mbox{where } v>2,\nonumber
\end{eqnarray}
thus $\bar{F}_{\beta}(t)=2\cdot \frac {\Gamma((v+1)/2)}{\sqrt{v\pi}\Gamma(v/2)}v^{(v+1)/2}\cdot t^{-v}\left[\frac 1 {v} -\frac {v(v+1)}{2(v+2)}t^{-2}+O(t^{-4})\right]$ as $t\to\infty$.
\end{description}
\begin{remark}
A tail distribution $\bar{F}$ is said to be the second order regularly varying (see, e.g., de Haan and Stadtm\"{u}ller (1996)~\cite{DeHaan-Stadtmuller:1996},
Geluk et al. (1997) \cite{Geluk-DeHaan-Resnick-Starica:1997} or Resnick (2007) \cite{Resnick:2008}) with the first order parameter $-\sigma<0$ and the second order parameter $\varrho<0$, written as $\bar{F}\in 2RV(-\sigma,\varrho)$, if there exists an ultimately positive or negative auxiliary function $A(t)$ with $\lim_{t\to\infty}A(t)= 0$, and a constant $c\neq 0$ such that
\[
\lim_{t\to\infty}\frac{\bar{F}(xt)/\bar{F}(t)-x^{-\sigma}}{A(t)}= c x^{-\sigma}\frac {x^{\varrho}-1} {\varrho}, \quad\mbox{ for all } x>0,
\]
One can easily check that the distribution tail $\bar{F_{\beta}}$ in Assumption A belongs to $2RV(-a,-h)$, by setting $A(t)=t^{-h}L_0(t)$ and $c=-h r_2/r_1$.
\end{remark}

\begin{remark}
For characterizing the first order asymptotic properties, say for the queueing length distribution, it is enough to make the first order asymptotic assumption for the service time distribution $F_{\beta}$, e.g., $\bar{F}_{\beta}(t) \sim t^{-a}L(t)$, where $L(t)$ can be any slowly varying function at $\infty$. However, for the second order approximation, it is necessary  to further specify the second term in the asymptotic expansion for the service time distribution. Since the second order regular variation is a standard definition assumed in the literature for second order approximations, readers could also expect that the same assumption can be made in our paper. Our Assumption~A is indeed an effort in this direction. To see the relationship between Assumption~A and the standard second order regular variation (2RV), in the revised version of our paper, we followed Hua and Joe (2011)~\cite{Hua-Joe:2011} to have the following equivalent condition to 2RV:
\begin{equation} \label{A}
    \overline{F}(t)=r_1 t^{-a}f(t)
\end{equation}
with $r_1>0$ and $\lim_{t\to\infty}f(t)=1$ and $|1-f(t)|$ being a regularly varying function of index $-h$  ($h>0$). It is then not difficult to see that our Assumption~A can be also written in the form of (\ref{A}), where $1-f(t)=kt^{-h}L_0(t)$. Therefore, the difference between the standard 2RV (where $|1-f(t)|$ is regularly varying) and Assumption~A (where $1-f(t)$ is regularly varying) is considered very minor.

The assumption $a>2$ implies that the second moment $\beta_2$ of the service time is finite, which can be potentially weakened (see Remark~\ref{rem:4.2} in Section~\ref{sec:4} for more information).
\end{remark}

The main result in this paper is given in the following theorem:
\begin{theorem}
\label{the:main}
For the $M/G/1$ retrial queue, under Assumption~A, we have the following second order asymptotic expansion:
\begin{equation}  \label{the:main-2}
P\{L_{\mu}>j\}=\frac {\lambda^a r_1} {(a-1)(1-\rho)} j^{-a+1}+\Delta_L(j),
\end{equation}
where as $j\to\infty$,
\begin{eqnarray}
    \Delta_L(j) &=& \left \{ \begin{array}{ll}
     \displaystyle \left[c_L+\frac {\lambda r_2} {a(1-\rho)} L_0(j)\right]\lambda^a j^{-a}+o(j^{-a}), & \mbox{for $h=1$}, \\
\displaystyle \frac {\lambda^{a+h} r_2} {(a+h-1)(1-\rho)} j^{-a-h+1}L_0(j) +o(j^{-a-h+1}), & \mbox{for $0<h<1$}, \\
\displaystyle c_L\lambda^a j^{-a}+o(j^{-a}), & \mbox{for $h>1$},
\end{array} \right.
\nonumber\\
c_L&=&\frac {r_1 (a/2-1) } {1-\rho}+\frac { \lambda r_1(\rho+1/a)} {\mu (1-\rho)^2}+\frac {\lambda^2\beta_2 r_1} {(1-\rho)^2 }.\nonumber
\end{eqnarray}
\end{theorem}

%{\color{red} As an immediate consequence of the above theorem, the second order approximations can be obtained for the standard (non-retrial) $M/G/1$ queue by %letting $\mu=\infty$ in the above theorem.}

The second order asymptotic expansion given in Theorem~\ref{the:main} can be viewed as a refined result of the equivalence theorem for retrial queues
under the assumption of a heavy-tailed service time (e.g., \cite{Shang-Liu-Li:2006}, \cite{Yamamuro:2012} and \cite{Masuyama:2014}).
%see for example, \cite{Shang-Liu-Li:2006}.
Recently, another refinement of the equivalence theorem was provided in \cite{Liu-Min-Zhao:2017}, which gives the first order approximation to the difference
$P\{L_\mu > j\} -P\{L_\infty >j\}$. However, the second order approximation to $L_\mu$ is not a consequence of this refinement (see discussions in Section~\ref{sec:concluding} for details).

The rest of the paper is organized as follows: Section~\ref{sec:2} provides preliminaries to facilitate our main analysis; Section~\ref{sec:3} and Section~\ref{sec:4} contain key results for proving the main theorem; Section~\ref{sec:5} completes the proof to the main result (Theorem~\ref{the:main}); and the final section, Section~\ref{sec:concluding}, contains concluding remarks and numerical results.

\section{Preliminary}
\label{sec:2}

In this section, several stochastic decompositions, as sums of independent or i.i.d. r.v.'s, will be introduced. Notations and properties for the involved r.v.'s will be discussed. These concepts will facilitate the process in later sections to our main result.

For the stable $M/G/1$ retrial queueing system defined earlier, let $N_{orb}$ be the number of repeated customers in the orbit and let $I_{ser}=1$ or 0 according to a busy server or an idle server, respectively. Let $R_{\mu}$ be a r.v.
taking nonnegative integer values with the probability generating function (GF) defined by $Ez^{R_{\mu}}\stackrel{\rm def}{=}E(z^{N_{orb}}|I_{ser}=0)$.
It follows from \cite{Falin-Templeton:1997} (pp.14--16) that $P\{I_{ser}=0\}=1-\rho$ and
\begin{equation}\label{Rmu1}
Ez^{R_{\mu}} = \exp\left\{-\frac {\lambda} {\mu}\int_z^1\frac {1-\beta(\lambda-\lambda u)} {\beta(\lambda-\lambda u)-u} du\right\}.
\end{equation}

It is well known that for the $M/G/1$ retrial queue, the total number $L_{\mu}$ of customers in the system can be written as the sum of two independent random variables (see, p.15 in \cite{Falin-Templeton:1997}): the total number $L_{\infty}$ of customers in the corresponding standard $M/G/1$ queueing system (without retrial) and $R_{\mu}$, i.e.,
\begin{equation}\label{L=L+D}
    L_{\mu}\stackrel{\rm d}{=}L_{\infty}+R_{\mu},
\end{equation}
where the symbol
$\stackrel{\rm d}{=}$ means equality in probability distribution. Such a symbol will be used throughout this paper.
The equality (\ref{L=L+D}) can be verified easily because
\begin{eqnarray}\label{p0(z)+zp1(z)}
Ez^{L_{\mu}} &=& \sum_{n=0}^{\infty}z^nP\{I_{ser}=0,N_{orb}=n\}+\sum_{n=0}^{\infty}z^{n+1}P\{I_{ser}=1,N_{orb}=n\}\nonumber\\
&=&p_0(z)+zp_1(z),
\end{eqnarray}
where $p_i(z)\stackrel{\rm def}{=}\sum_{n=0}^{\infty}z^nP\{I_{ser}=i,N_{orb}=n\}$, $i=0,1$, are explicitly expressed (e.g., pp.9--10 in \cite{Falin-Templeton:1997}). The expressions for $p_i(z)$, together with (\ref{p0(z)+zp1(z)}), leads to $Ez^{L_{\mu}}=Ez^{L_{\infty}}\cdot Ez^{R_{\mu}}$, since
\begin{eqnarray}
Ez^{L_{\mu}} &=& \frac {(1-\rho)(1-z)} {\beta(\lambda-\lambda z)-z}\cdot \beta(\lambda-\lambda z)\cdot\exp\left\{-\frac {\lambda} {\mu}\int_z^1\frac {1-\beta(\lambda-\lambda u)} {\beta(\lambda-\lambda u)-u} du\right\},\\
Ez^{L_{\infty}} &=& \frac {(1-\rho)(1-z)} {\beta(\lambda-\lambda z)-z}\cdot \beta(\lambda-\lambda z),\label{without-retrial}
\end{eqnarray}
which results in (\ref{L=L+D}).

The stochastic decomposition (\ref{L=L+D}) is often used to establish the asymptotic equivalence:
\begin{equation}\label{Lmu=Linfty}
P\{L_{\mu}>j\}\sim P\{L_{\infty}>j\}\quad\mbox{ as }j\to\infty,
\end{equation}
under the assumption of a heavy-tailed service time (e.g., \cite{Shang-Liu-Li:2006}, \cite{Yamamuro:2012} and \cite{Masuyama:2014}), which reveals the first order asymptotic behaviour of $P\{L_{\mu}>j\}$.
One should notice that, in the first order approximation to $P\{L_{\mu}>j\}$, $P\{R_{\mu}>j\}$ is dominated by $P\{L_{\infty}>j\}$, so a detailed (first order) asymptotic behaviour of $P\{R_{\mu}>j\}$ is not required, and the first order asymptotic behaviour of $P\{L_{\infty}>j\}$ is sufficient.
However, as will be shown in later sections, both the first order asymptotic behaviour of $P\{R_{\mu}>j\}$ and the second order asymptotic expansion of $P\{L_{\infty}>j\}$ are required for determining the second term in the second order asymptotic expansion of $P\{L_{\mu}>j\}$.

To this end, we first rewrite (\ref{Rmu1}). Let
\begin{eqnarray}
\psi &=&\frac{\rho}{\mu(1-\rho)}, \label{psi} \\
\kappa(s)&=&\frac{1-\rho} {\beta_1}\cdot\frac{1-\beta(s)} {s-\lambda+\lambda\beta(s)},\label{K(u)}\\
\tau(s)&=&\exp\left\{-\psi\int_0^s\kappa(u)du\right\}.\label{tau00}
\end{eqnarray}
It is easy to see, from (\ref{Rmu1}) and (\ref{psi})--(\ref{tau00}), that
\begin{equation}
    Ez^{R_{\mu}} = \exp\left\{-\frac {\lambda} {\mu}\int_0^{\lambda-\lambda z}\frac{1-\beta(s)} {s-\lambda+\lambda\beta(s)}ds\right\}=\tau(\lambda-\lambda z).\label{D^{(0)}}
\end{equation}

We now show that both $\kappa(s)$ and $\tau(s)$ are the LSTs of two probability distributions, respectively. For the first assertion, let $F_{\beta}^{(e)}(x)$ be the so-called equilibrium distribution of $F_{\beta}(x)$,
which is defined as $F_{\beta}^{(e)}(x)= \beta_1^{-1}\int_0^{x}(1-F_{\beta}(t))dt$. The LST $\beta^{(e)}(s)$ of $F_{\beta}^{(e)}(x)$ can be written as
$\beta^{(e)}(s)=(1-\beta(s))/(\beta_1 s)$.
From (\ref{K(u)}), we have
\begin{equation}\label{kappa(s)-2}
\kappa(s) = \frac{(1-\rho)\beta^{(e)}(s)}{1-\rho\beta^{(e)}(s)}=\sum_{k=1}^{\infty}(1-\rho)\rho^{k-1}(\beta^{(e)}(s))^k.
\end{equation}

\begin{remark}\label{kappa-geo-sum}
Define $T_{\kappa}$ to be a geometric sum of i.i.d. r.v.'s $T_{\beta,j}^{(e)}$, $j\ge 1$, each with the distribution $F_{\beta}^{(e)}(x)$; or more specifically,
\begin{equation}\label{K{circ}}
T_{\kappa}\stackrel{\rm d}{=}T_{\beta,1}^{(e)}+T_{\beta,2}^{(e)}+\cdots+T_{\beta,J}^{(e)},
\end{equation}
where $P(J=j)=(1-\rho)\rho^{j-1}$, $j\ge 1$ and $J$ is independent of $T_{\beta,j}^{(e)}$ for $j\ge 1$. Then, immediately from (\ref{kappa(s)-2}), $\kappa(s)$ can be viewed as the LST of the distribution function $F_{\kappa}(\cdot)$ of the r.v. $T_{\kappa}$.
\end{remark}

For the second assertion that $\tau(s)$ is the LST of a probability distribution on $[0,\infty)$, denoted by $F_{\tau}(x)$, we use mathematical induction.  By Theorem~1 in Feller (1991)~\cite{Feller1971} (see p.439),
it is true as long as $\tau(s)$ is completely monotone, i.e., $\tau(s)$ possesses derivatives $\tau^{(n)}(s)$ of all orders such that $(-1)^{n}\tau^{(n)}(s)\ge 0$ for $s> 0$, and $\tau(0)=1$. First, it is clear from (\ref{tau00}) that $\tau(0)=1$ and
\begin{equation}\label{tau-probab-dist}
\tau^{(1)}(s)=-\psi\cdot\tau(s)\kappa(s).
\end{equation}
We then proceed with the mathematical induction on $n$. Obviously, $-\tau^{(1)}(s)\ge 0$ for $s> 0$. Next, let us make the induction hypothesis that $(-1)^{k}\kappa^{(k)}(s)\ge 0$ for $s> 0$ and all $k=1,2,\ldots, n$. Taking derivatives $n$ times on both sides of (\ref{tau-probab-dist}), we get
\begin{equation}
\tau^{(n+1)}(s)=-\psi\cdot\sum_{i=0}^{n}\binom{n}{i}\tau^{(i)}(s) \kappa^{(n-i)}(s).
\end{equation}
Therefore,
\begin{equation}
(-1)^{n+1}\tau^{(n+1)}(s)=\psi\cdot\sum_{i=0}^{n}\binom{n}{i}\left[(-1)^{i}\tau^{(i)}(s)\right]\cdot \left[(-1)^{n-i}\kappa^{(n-i)}(s)\right]\ge 0\quad \mbox{for }s> 0.
\end{equation}
By Remark~\ref{kappa-geo-sum}, $\kappa(s)$ is the LST of probability distribution $F_{\kappa}(\cdot)$, hence $(-1)^{k}\kappa^{(k)}(s)\ge 0$ for $s> 0$ and $k=1,2,\ldots$, which, together with the induction hypothesis, completes the proof for $k=n+1$.

\begin{remark}\label{tau-compound-possion}
Let $T_{\tau}$ be a r.v. having the distribution $F_{\tau}(x)$. Since $\tau(s)$ is the LST of the probability distribution $F_{\tau}(x)$,
the expression $Ez^{R_{\mu}}=\tau(\lambda-\lambda z)$  in (\ref{D^{(0)}}) implies that $R_{\mu}$ can be regarded as the number of Poisson arrivals at rate $\lambda$ within a random time $T_{\tau}$.
\end{remark}

In the following, we consider a stochastic decomposition for $L_\infty$.
Note that (\ref{without-retrial}) can be rewritten as
\begin{equation}\label{Ez-L-infty}
Ez^{L_{\infty}} = \theta(\lambda-\lambda z)\cdot\beta(\lambda-\lambda z),
\end{equation}
where
\begin{equation}\label{theta(s)}
\theta(s)= 1-\rho +\rho\kappa(s).
\end{equation}
This implies that $\theta(s)$ can be viewed as the LST of the probability distribution function $F_{\theta}(t)$ of a r.v. $T_{\theta}$, defined by
\begin{eqnarray}\label{def-T-theta}
T_{\theta}&\stackrel{\rm def}{=}&\left\{\begin{array}{ll}
0, &\mbox{ with probability }1-\rho,\\
T_{\kappa}, &\mbox{ with probability }\rho.
\end{array}
\right.
\end{eqnarray}
%where $T_{\kappa}$ is a r.v. having probability distribution function $F_{\kappa}(t)$ with LST $\kappa(s)$.

\begin{remark}\label{remark-DLL}
By (\ref{Ez-L-infty}), with the same argument as that in Remark~\ref{tau-compound-possion}, one can interpret $L_{\infty}$ as the number of Poisson arrivals at rate $\lambda$ within a random time $T_{\theta}+T_{\beta}$, where $T_{\theta}$ and $T_{\beta}$ are independent, and  $L_{\mu}$ can also be interpreted in a similar fashion.
\end{remark}

More precisely, the interpretations in Remarks~\ref{tau-compound-possion} and \ref{remark-DLL} are restated in the following lemma.

\begin{lemma} \label{lem:interpretations}
Let $N_t$ be a Poisson process with rate $\lambda$, which is independent of the r.v.'s $T_{\kappa}$, $T_{\beta}$ and $T_{\tau}$ (defined in the above discussions). Then, $R_{\mu}$, $L_{\infty}$ and $L_{\mu}$ can be expressed as follows:
\begin{eqnarray}
R_{\mu}&\stackrel{\rm d}{=}&N_{T_{\tau}},\label{DLL-0}\\
L_{\infty}&\stackrel{\rm d}{=}&N_{T_{\theta}+T_{\beta}}\stackrel{\rm d}{=}N_{T_{\theta}}+N_{T_{\beta}},\label{DLL-1} \\
L_{\mu}&\stackrel{\rm d}{=}&N_{T_{\theta}+T_{\beta}+T_{\tau}}\stackrel{\rm d}{=}N_{T_{\theta}}+N_{T_{\beta}}+N_{T_{\tau}}.\label{DLL-2}
\end{eqnarray}
\end{lemma}

%\begin{remark}
%As a matter of fact, (\ref{DLL-1}) is the well known Functional Little's Law (FLL) for the standard $M/G/1$ queue (see, e.g., Asmussen et al. (1999) \cite{Asmussen-Klupperlberg-Sigman:1999}).
%\end{remark}

Our asymptotic analysis (in the following sections) is based on the assumption that
the service time $T_{\beta}$ has a so-called regularly varying tail. We adopt the following definitions.

\begin{definition}[e.g., Bingham et al. (1989) \cite{Bingham:1989}]
\label{Definition 3.1}
A measurable function $U:(0,\infty)\to (0,\infty)$ is regularly varying at $\infty$ with index $\sigma\in(-\infty,\infty)$ (written $U\in R_{\sigma}$) iff $\lim_{t\to\infty}U(xt)/U(t)=x^{\sigma}$ for all $x>0$. If $\sigma=0$ we call $U$ slowly varying, i.e., $\lim_{t\to\infty}U(xt)/U(t)=1$ for all $x>0$.
\end{definition}

\begin{definition}[e.g., Foss, Korshunov and Zachary (2011)~\cite{Foss2011}]
\label{Definition 3.2}
A distribution $F$ on $(0,\infty)$ belongs to the class of subexponential distribution (written $F\in \mathcal S$) if $\lim_{t\to\infty}\bar{F}^{*2}(t)/\bar{F}(t)=2$, where $\bar{F}=1-F$ and $\bar{F}^{*2}$ denotes the second convolution of $\bar{F}$.
\end{definition}

It is well known that for a distribution $F$ on $(0,\infty)$, if $\bar{F}(t)\sim t^{-\alpha}L(t)$, where $\alpha\ge 0$ and $L(t)$ is a slowly varying function at $\infty$, then $F\in \mathcal S$
(see, e.g., Embrechts, Kluppelberg and Mikosch (1997)~\cite{Embrechts1997}).

\bigskip
For convenience to the readers, main notations and key quantities are summarized in the following list:

\begin{description}
\item[$\lambda$]: Poisson arrival rate;

\item[$\mu$]: Poisson retrial rate;

\item[$T_\beta$] and {\bf $F_\beta(\cdot)$\ }: Service time r.v. and its distribution, respectively;

\item[$\beta_n$]: $n$th moment of $T_\beta$;

\item[$\beta(s)$]: LST of $F_\beta(\cdot)$;

\item[$\rho=\lambda \beta_1$]: Traffic intensity of the $M/G/1$ retrial queue;

\item[$N_{orb}$]: Number of the repeated customers in the orbit;

\item[$I_{ser}$]: State of the server, taking values 1 or 0 according to busy or idle, respectively;

\item[$R_\mu$]: r.v. having distribution $P\{R_\mu=j\} {=} P\{N_{orb}=j|I_{ser}=0\}$;

\item[$L_\mu$]: Total number of customers in the $M/G/1$ retrial queueing system;

\item[$L_\infty$]: Total number of customers in the corresponding standard (non-retrial) $M/G/1$ queueing system;

\item[$L_\mu \stackrel{\rm d}{=} L_\infty + R_\mu$]: Stochastic decomposition;

\item[$\psi$]: A constant given in (\ref{psi});

\item[$F_{\beta}^{(e)}(x)$]:  Equilibrium distribution of $F_{\beta}(\cdot)$;

\item[$\beta^{(e)}(s)$]: LST of $F_{\beta}^{(e)}(\cdot)$;

\item[$T^{(e)}_{\beta,j}$]: i.i.d. r.v.'s having distribution $F_{\beta}^{(e)}(\cdot)$;

\item[$J$]: Geometric r.v. with parameter $\rho$, independent of $T^{(e)}_{\beta,j}$;

\item[$T_\kappa \stackrel{\rm d}{=} \sum_{j=1}^J T^{(e)}_{\beta,j}$] and {\bf $F_\kappa(\cdot)$}\ : Geometric sum of $T^{(e)}_{\beta,j}$ and its distribution;

\item[$\kappa(s)$]: Defined in (\ref{K(u)}) and proved to be the LST of $F_\kappa(\cdot)$;

\item[$\tau(s)$] and  {\bf $F_\tau(\cdot)$\ }: $\tau(s)$ is defined in (\ref{tau00}) and proved to be the LST of a probability distribution $F_\tau(\cdot)$;

\item[$T_\tau$]: r.v. having distribution $F_\tau(\cdot)$;

\item[$T_\theta$] and {\bf $F_\theta(\cdot)$}\ : r.v. taking values 0 or $T_\kappa$ with probability $1-\rho$ or $\rho$, respectively, and its distribution;

\item[$\theta(s)$]: LST of $F_\theta(\cdot)$;

\item[$L(t)$] and  {\bf $L_0(t)$\ }:  Slowly varying functions at $\infty$;

\item[$N_t$]: Poisson process with rate $\lambda$;

\item[$R_{\mu} \stackrel{\rm d}{=} N_{T_{\tau}},$
$L_{\infty} \stackrel{\rm d}{=} N_{T_{\theta}+T_{\beta}}\stackrel{\rm d}{=}N_{T_{\theta}}+N_{T_{\beta}}$] and
 {\bf $L_{\mu} \stackrel{\rm d}{=} N_{T_{\theta}+T_{\beta}+T_{\tau}}\stackrel{\rm d}{=}N_{T_{\theta}}+N_{T_{\beta}}+N_{T_{\tau}}$\ }: See Lemma~\ref{lem:interpretations}.

\end{description}

\section{Asymptotic analysis for tail probability of $T_{\tau}$}
\label{sec:3}

In this section, we provide an asymptotic property for the tail probability of the r.v. $T_{\tau}$, which will be stated in Theorem~\ref{theorem-T-tail}.
This is the first order asymptotic expansion of $P\{T_{\tau}>t\}$, which holds under a weaker condition (a relaxation of Assumption~A): $P\{T_{\beta}>t\} \sim t^{-a}L(t)$ as $t\to\infty$ where $a>1$.

\begin{lemma}[pp.580--581 in \cite{Embrechts1997}]
\label{Embrechts-compound-geo}
Let $N$ be a r.v. with $P\{N=k\}=(1-\rho)\rho^{k-1}$, $k=1,2,\ldots$, and
$\{Y_k\}_{k=1}^{\infty}$ be a sequence of non-negative i.i.d. r.v.'s having a common subexponential distribution $F$.
Define $S_n=\sum_{k=1}^n Y_k$. Then
\begin{equation}
P\{S_N > t\} \sim  \frac {1} {1-\rho} (1-F(t)),\quad t\to \infty.
\end{equation}
\end{lemma}

Suppose that $P\{T_{\beta}>t\} \sim t^{-a}L(t)$ as $t\to\infty$ where $a>1$, by Karamata's theorem (e.g., p.28 in \cite{Bingham:1989}), we have $\int_t^{\infty}(1-F_{\beta}(x))dx\sim (a-1)^{-1} t^{-a+1} L(t)$,
which implies $1-F_{\beta}^{(e)}(t) \sim ((a-1)\beta_1)^{-1}  t^{-a+1} L(t)$, $t\to\infty$.
By Remark~\ref{kappa-geo-sum} and applying Lemma~\ref{Embrechts-compound-geo}, we have
\begin{equation}\label{P{K>x}}
P\{T_{\kappa}>t\} \sim c_{\kappa}\cdot t^{-a+1} L(t),\quad t\to\infty,
\end{equation}
where
\begin{equation}\label{c-K}
c_{\kappa} = \frac 1 {(1-\rho)(a-1)\beta_1}.
\end{equation}

Note that $T_{\tau}$ has the distribution function $F_{\tau}$ defined in terms of its LST $\tau(s)$ in (\ref{tau00}),
which is, therefore, determined by the distribution function $F_{\kappa}$ of $T_{\kappa}$.
In the following theorem, we present the asymptotic tail probability of $T_{\tau}$.
\begin{theorem}\label{theorem-T-tail}
Suppose that $P\{T_{\beta}>t\} \sim t^{-a}L(t)$ as $t\to\infty$ where $a>1$. Then
\begin{equation}
P\{T_{\tau}>t\} \sim (1-1/a)c_{\kappa}\psi\cdot t^{-a}L(t),\quad t\to\infty,\label{main}
\end{equation}
where $\psi$ and $c_{\kappa}$ are expressed in (\ref{psi}) and (\ref{c-K}), respectively.
\end{theorem}

To prove Theorem \ref{theorem-T-tail}, let us list some notations and preliminary results, which will be used.
Let $F(x)$ be any distribution on $[0,\infty)$ with the LST $\phi(s)$.
We denote the $n$th moment of $F(x)$ by $\phi_n$, $n\ge 0$.
It is well known (for example, Proposition~8.44 in Breiman~\cite{Breiman:1992}) that if $\phi_n<\infty$, then
\begin{equation}
\phi(s)=\sum_{k=0}^{n}\frac{\phi_k}{k!}(-s)^k + o(s^n),\quad n\ge 0.
\end{equation}
Next, if $\phi_n<\infty$, we introduce the notation $\phi_n(s)$ and $\widehat{\phi}_n(s)$, defined by
\begin{eqnarray}
\phi_n(s)&\stackrel{\rm def}{=}&(-1)^{n+1}\left\{\phi(s)-\sum_{k=0}^{n}\frac{\phi_k}{k!}(-s)^k\right\},\quad n\ge 0,\label{phi1}\\
\widehat{\phi}_n(s)&\stackrel{\rm def}{=}&\phi_n(s)/s^{n+1},\quad n\ge 0.\label{phi2}
\end{eqnarray}
It follows that if $\phi_n<\infty$, then for $n\ge 1$,
\begin{eqnarray}
\lim_{s\downarrow 0}\widehat{\phi}_{n-1}(s)&=&\phi_n/n!,\\
s\widehat \phi_{n}(s)&=&\frac 1 {n!}\phi_{n}-\widehat \phi_{n-1}(s).
\end{eqnarray}

In addition, if $\phi_n<\infty$, let us define a sequence of functions $F_k$ recursively by:
$F_1(t)=F(t)$ and
\begin{equation}\label{definition-Fk(t)}
1-F_{k+1}(t) \stackrel{\rm def}{=} \int_t^{\infty}(1-F_k(x))dx,\quad k=1,2,\ldots,n.
\end{equation}
It is not difficult to check that $1-F_{k+1}(0)=\phi_k/k!$ and $F_{k+1}(t)$ has the LST $\widehat \phi_{k-1}(s)$. Namely,
\begin{equation}
\widehat \phi_{k-1}(s) = \int_0^{\infty}e^{-st}(1-F_{k}(t))dt,\quad k=1,2,\ldots,n.\label{laplace-phi(t)}
\end{equation}
\begin{lemma}[pp.333-334 in \cite{Bingham:1989}]
\label{Cohen}
Assume that $n<d<n+1$, $n\in\{0,1,2,\ldots\}$. Then the following are equivalent:
\begin{eqnarray}
1-F(t) &\sim&  t^{-d} L(t),\quad t\to\infty,\\
\phi_n(s) &\sim& \frac{\Gamma(d-n)\Gamma(n+1-d)} {\Gamma(d)} s^{d}L(1/s),\quad s\downarrow 0.
\end{eqnarray}
\end{lemma}

In the following, we will divide the proof of Theorem \ref{theorem-T-tail} into two parts, depending on whether $a$ is an integer or not.
First let us rewrite (\ref{tau00}) as follows:
\begin{equation}
\tau(s) = 1 -\psi\int_0^s\kappa(u)du +\sum_{k=2}^{\infty}{(-\psi)^k\over k!}\left(\int_0^s\kappa(u)du\right)^k.\label{tau1}
\end{equation}

\subsection{A proof of Theorem~\ref{theorem-T-tail} for non-integer $a>1$}

Suppose that $m<a<m+1$, $m\in\{1,2,\ldots\}$. By (\ref{P{K>x}}), $P\{T_{\kappa}>t\}\sim c_{\kappa}\cdot t^{-a+1}L(t)$. So, $\kappa_{m-1}<\infty$ and $\kappa_{m}=\infty$.
Define $\kappa_{m-1}(s)$ in a manner similar to that in (\ref{phi1}). By Lemma~\ref{Cohen} ,
\begin{equation}
\kappa_{m-1}(s) \sim \frac{\Gamma(a-m)\Gamma(m+1-a)} {\Gamma(a-1)} c_{\kappa} s^{a-1}L(1/s),\quad s\downarrow 0.
\end{equation}
By Karamata's theorem (p.28 in \cite{Bingham:1989}),
\begin{equation}\label{int-kappa-{m-1}(u)du}
\int_0^s\kappa_{m-1}(u)du \sim \frac{\Gamma(a-m)\Gamma(m+1-a)} {\Gamma(a-1)a} c_{\kappa} s^{a}L(1/s),\quad s\downarrow 0.
\end{equation}
Next, we present a relationship between $\tau_m(s)$ and $\kappa_{m-1}(s)$. By the definition of $\kappa_{m-1}(s)$,
\begin{equation}
\kappa(s) = \sum_{k=0}^{m-1}\frac{\kappa_k}{k!}(-s)^k+(-1)^{m}\kappa_{m-1}(s),
\end{equation}
where $\kappa_{m-1}(s)=o(s^{m-1})$ and $\kappa_{m-1}(s)/s^m\to\infty$ as $s\downarrow 0$. Hence,
\begin{equation}
\int_0^s\kappa(u)du = -\sum_{k=1}^{m}\frac{\kappa_{k-1}}{k!}(-s)^k+(-1)^{m}\int_0^s\kappa_{m-1}(u)du,\label{kappa1}
\end{equation}
where $\int_0^s\kappa_{m-1}(u)du=o(s^m)$ and $\int_0^s\kappa_{m-1}(u)du/s^{m+1}\to\infty$ as $s\downarrow 0$.
\newline

From (\ref{tau1}) and (\ref{kappa1}), there are constants $\{v_k;\ k=0,1,2,\ldots,m\}$ satisfying
\begin{equation}
\tau(s) = \sum_{k=0}^{m}v_k (-s)^k+(-1)^{m+1}\psi\int_0^s\kappa_{m-1}(u)du + O(s^{m+1}),\quad s\downarrow 0.\label{tau(s)-new}
\end{equation}
Define $\tau_m(s)$ in a manner similar to that in (\ref{phi1}). By (\ref{tau(s)-new}),
\begin{equation}
\tau_m(s) = \psi\int_0^s\kappa_{m-1}(u)du + O(s^{m+1}) \sim \psi\int_0^s\kappa_{m-1}(u)du,\quad s\downarrow 0.\label{case1result}
\end{equation}
By (\ref{int-kappa-{m-1}(u)du}) and (\ref{case1result}),
\begin{equation}\label{tau_m(s)}
\tau_m(s) \sim \frac{\Gamma(a-m)\Gamma(m+1-a)} {\Gamma(a)} \cdot\frac {a-1} {a} c_{\kappa}\psi s^{a}L(1/s),\quad s\downarrow 0.
\end{equation}
Applying Lemma \ref{Cohen},
\begin{equation}
P\{T_{\tau}>t\} \sim \frac {a-1} {a} c_{\kappa}\psi t^{-a}L(t),\quad t\to\infty,\label{(a>1)4.2-17}
\end{equation}
which completes the proof of Theorem \ref{theorem-T-tail} for non-integer $a>1$.

\subsection{A proof of Theorem \ref{theorem-T-tail} for integer $a>1$}

Suppose that $a=m\in\{2,3,\ldots\}$.  By (\ref{P{K>x}}), $P\{T_{\kappa}>t\}\sim c_{\kappa}\cdot t^{-m+1}L(t)$.
So, $\kappa_{m-2}<\infty$. Unfortunately, whether $\kappa_{m-1}$ is finite or not remains uncertain, which is determined essentially by whether
$\int_x^{\infty}t^{-1}L(t)dt$ is convergent or not. For this reason we need to sharpen our tools by
introducing the de Haan class $\Pi$ of slowly varying functions.
\begin{definition}[e.g., Bingham et al. (1989) \cite{Bingham:1989}]
A function $F:(0,\infty)\to (0,\infty)$ belongs to the de Haan class $\Pi$ at $\infty$ if there exists a function $H:(0,\infty)\to (0,\infty)$ such that
\begin{equation}
\lim_{t\uparrow \infty}\frac {F(xt)-F(t)}{H(t)} = \log x\quad\mbox{ for all }x>0,\label{deHaan1}
\end{equation}
where the function $H$ is called the auxiliary function of $F$.
%Similarly, $F(t)$ belongs to the class $\Pi$ at $0$ if $F(1/t)$ belongs to $\Pi$ at $\infty$, or equivalently, there exists a function $H:(0,\infty)\to %(0,\infty)$ such that
%\begin{equation}
%\lim_{s\downarrow 0}\frac {F(xs)-F(s)}{H(s)} = -\log x\quad\mbox{ for all }x>0.\label{deHaan2}
%\end{equation}
\end{definition}

Recall the definition of $F_{k}(t)$ given in (\ref{definition-Fk(t)}). Repeatedly using Karamata's theorem (p.27 in \cite{Bingham:1989}) and the monotone density theorem (p.39 in \cite{Bingham:1989}),
we know that $1-F(t)\sim t^{-n}L(t)$ is equivalent to $1-F_n(t)\sim t^{-1}L(t)/(n-1)!$, which in turn is equivalent to
$\int_0^t(1-F_n(x))dx\in \Pi$ with an auxiliary function that can be taken as $L(t)/(n-1)!$ (see for example, p.335 in \cite{Bingham:1989}).
By (\ref{laplace-phi(t)}),  $\int_0^t(1-F_n(x))dx$ has the LST $\widehat{\phi}_{n-1}(s)$. Applying Theorem~3.9.1 in \cite{Bingham:1989} (pp.172--173), we have the following lemma.
\begin{lemma}
\label{Lemma 4.5new}
Let $F(x)$ be a probability distribution function and $n\in\{1,2,\ldots\}$. Then, the following two statements are equivalent:
\begin{flalign}
\begin{split}
\mbox{(i)}&\quad 1-F(t)\sim t^{-n}L(t),\quad t\to\infty; \label{deHaan3}
\end{split}&\\
\begin{split}
\mbox{(ii)}&\quad \lim_{s\downarrow 0}\frac {\widehat{\phi}_{n-1}(xs)-\widehat{\phi}_{n-1}(s)}{L(1/s)/(n-1)!}=-\log x, \quad\mbox{ for all }x>0.\label{deHaan4}
\end{split}&
\end{flalign}
%where $\widehat{\phi}_{n-1}(s)$ is defined in (\ref{phi2}).
\end{lemma}

Since $\kappa_{m-2}<\infty$, we can define $\kappa_{m-2}(s)$ in a manner similar to that in (\ref{phi1}), such that,
\begin{equation}
\kappa(s) = \sum_{k=0}^{m-2}\frac{\kappa_k}{k!}(-s)^k+(-1)^{m-1}\kappa_{m-2}(s),\label{case2result1}
\end{equation}
where $\kappa_{m-2}(s)=o(s^{m-2})$ as $s\downarrow 0$. Therefore,
\begin{equation}
\int_0^s\kappa(u)du = -\sum_{k=1}^{m-1}\frac{\kappa_{k-1}}{k!}(-s)^k+(-1)^{m-1}\int_0^s\kappa_{m-2}(u)du,\label{kappa2}
\end{equation}
where $\int_0^s\kappa_{m-2}(u)du=o(s^{m-1})$ as $s\downarrow 0$.
\newline

It follows from (\ref{tau1}) and (\ref{kappa2}) that for some constants $\{w_k;\ k=0,1,2,\ldots,m\}$,
\begin{equation}
\tau(s) = \sum_{k=0}^{m}w_k (-s)^k+(-1)^{m}\psi\int_0^s\kappa_{m-2}(u)du +o(s^{m}).
\end{equation}
Defining $\widehat{\tau}_{m-1}(s)$ in a manner similar to that in (\ref{phi2}), we have
\begin{equation}
\widehat{\tau}_{m-1}(s) = w_m+ \frac {\psi} {s^m}\int_0^s u^{m-1}\widehat{\kappa}_{m-2}(u)du + o(1),\label{case2tau-2}
\end{equation}
which immediately gives,
\begin{eqnarray}
\widehat{\tau}_{m-1}(xs)&=&w_m+ \frac {\psi} {(xs)^m}\int_0^{xs} u^{m-1}\widehat{\kappa}_{m-2}(u)du + o(1)\nonumber\\
&=&w_m+ \frac {\psi} {s^m}\int_0^{s} u^{m-1}\widehat{\kappa}_{m-2}(xu)du + o(1).\label{case2tau-3}
\end{eqnarray}
By (\ref{case2tau-2}) and (\ref{case2tau-3}),
\begin{equation}
\widehat{\tau}_{m-1}(xs)-\widehat{\tau}_{m-1}(s) = \frac {\psi} {s^m}\int_0^{s} u^{m-1}\left(\widehat{\kappa}_{m-2}(xu)-\widehat{\kappa}_{m-2}(u)\right)du + o(1).\label{case2tau-4}
\end{equation}
Note that $P\{T_{\kappa}>t\}\sim c_{\kappa}\cdot t^{-m+1}L(t)$ and $\kappa_{m-2}<\infty$. Define $\widehat{\kappa}_{m-2}(s)$ in a manner similar to that in (\ref{phi2}). By Lemma~\ref{Lemma 4.5new}, we obtain
\begin{equation}
\widehat{\kappa}_{m-2}(xu)-\widehat{\kappa}_{m-2}(u)\sim -(\log x) c_{\kappa} L(1/u)/(m-2)!\quad \mbox{ as}\ u\downarrow 0.\label{case2kappa-5}
\end{equation}
By Karamata's theorem (p.28 in \cite{Bingham:1989}), we know
\begin{equation}
\int_0^s u^{m-1}\left(\widehat{\kappa}_{m-2}(xu)-\widehat{\kappa}_{m-2}(u)\right)du\sim -(\log x) \frac {c_{\kappa}} {m} s^m L(1/s)/(m-2)!\quad \mbox{ as}\ s\downarrow 0.\label{case2kappa-5add}
\end{equation}
Therefore, by (\ref{case2tau-4}) and (\ref{case2kappa-5add}),
\begin{equation}
\lim_{s\downarrow 0}\frac {\widehat{\tau}_{m-1}(xs)-\widehat{\tau}_{m-1}(s)} {L(1/s)/(m-1)!} = -\frac {m-1} {m} c_{\kappa}\psi \log x.\label{case2tau-6}
\end{equation}
By applying Lemma~\ref{Lemma 4.5new}, we obtain from (\ref{case2tau-6}) that
\begin{equation}
P\{T_{\tau}>t\} \sim \frac {m-1} {m} c_{\kappa}\psi t^{-m}L(t),\quad t\to\infty,
\end{equation}
which completes the proof of Theorem \ref{theorem-T-tail} for integer $a=m\in \{2,3,\ldots\}$.

\section{Second order approximation for the tail probability of $T_{\theta}+T_{\beta}+T_{\tau}$}
\label{sec:4}

In this section, we will provide the second order asymptotic result (Theorem~\ref{the:tailT}) for the tail probability $P\{T_{\theta}+T_{\beta}+T_{\tau}>t\}$, which will be used in the next section to obtain the second order asymptotic expansion for the tail probability of $L_\mu$.
This requires further specifications on the second term of the asymptotic tail probability of service time $T_{\beta}$.
For this reason, Assumption A (stated in the introduction section) is made.

We first study the second order asymptotic behaviour of tail probability $P\{T_{\theta}>t\}$. Recall the definition of $T_{\theta}$ in (\ref{def-T-theta}), which immediately gives
\begin{equation}\label{6-3}
P\{T_{\theta}>t\} = \rho P\{T_{\kappa}>t\}.
\end{equation}
Refering to Remark \ref{kappa-geo-sum},
$T_{\kappa}$ can be viewed as a geometric sum of i.i.d. r.v.'s, which suggests that we need second order asymptotic properties of such a random sum.

Suppose that  $\{p_n\}_{n=0}^{\infty}$ is a discrete probability measure, and $G(t)=\sum_{n=0}^{\infty}p_n F^{*n}(t)$, $t\ge 0$.
\begin{lemma}[Theorem~2 in  \cite{Willekens1992}]
\label{Willekens-Th}
Suppose that $\sum_{n=0}^{\infty}p_nz^n$ is analytic at $z=1$. Let $F\in\mathcal S$ be a probability distribution function with finite mean $\phi_1$ and probability density function $f(t)=dF(t)/dt$.
Then the following assertions are equivalent:
\begin{flalign}
\begin{split}
\mbox{(i)}&\quad \lim_{t\to\infty}\frac {\bar{F}^{*2}(t)-2\bar{F}(t)}{f(t)}=2\phi_1; \label{H-1}
\end{split}&\\
\begin{split}
\mbox{(ii)}&\quad \lim_{t\to\infty}\frac {\bar{G}(t)-(\sum_{n=1}^{\infty}np_n)\bar{F}(t)}{f(t)}=\phi_1\sum_{n=2}^{\infty}n(n-1)p_n.\label{H-2}
\end{split}&
\end{flalign}
\end{lemma}

As a special case of Lemma \ref{Willekens-Th}, we set $p_0=0$ and $p_n=(1-\rho)\rho^{n-1},\ n\ge 1$. Then, $\sum_{n=1}^{\infty}np_n=1/(1-\rho)$ and $\sum_{n=2}^{\infty}n(n-1)p_n=2\rho/(1-\rho)^2$.
So, (\ref{H-2}) can be written as
\begin{equation}\label{Willekens-2nd-formula}
\bar{G}(t)=\frac 1 {1-\rho}\bar{F}(t) + \frac {2\rho\phi_1} {(1-\rho)^2} f(t) +o(f(t)).
\end{equation}

With the aid of (\ref{Willekens-2nd-formula}), we are ready to present the second order asymptotic expansion of the tail probability of $T_{\kappa}$.
Set $\bar{F}(t)=\bar{F}_{\beta}^{(e)}(t)$  in Lemma \ref{Willekens-Th}, i.e., $\bar{F}(t)=(1/\beta_1)\int_t^{\infty}\bar{F}_{\beta}(x)dx$ and $f(t)=(1/\beta_1)\bar{F}_{\beta}(t)$. It follows from Assumption A that
\begin{equation}
\bar{F}_{\beta}^{(e)}(t)=\frac {r_1} {\beta_1(a-1)} t^{-a+1}+\frac {r_2} {\beta_2(a+h-1)}t^{-a-h+1}L_0(t) (1+o(1)),
\end{equation}
which implies that $\bar{F}_{\beta}^{(e)}\in 2RV(-a+1,-h)$. By Proposition 3.8 in \cite{Liu-Mao-Hu:2017}, we know that (\ref{H-1}) holds for $\bar{F}(t)=\bar{F}_{\beta}^{(e)}(t)$.

Recall Remark \ref{kappa-geo-sum}. $T_{\kappa}$ can be viewed as a geometric sum of i.i.d. random variables with distribution $F_{\beta}^{(e)}(x)$.
Note that $\phi_1=(1/\beta_1)\int_0^{\infty}t \bar{F}_{\beta}(t)dt=\beta_2/(2\beta_1)$, where $\beta_2=\int_0^{\infty} t^2 dF_{\beta}(t)=\int_0^{\infty}2t\bar{F}_{\beta}(t)dt<\infty$. It follows from (\ref{Willekens-2nd-formula}) that
\begin{equation}\label{T-kappa-2nd-order}
P\{T_{\kappa}>t\}=\frac 1 {(1-\rho)\beta_1}\left(\int_t^{\infty}\bar{F}_{\beta}(x)dx\right)+\frac {\rho\beta_2} {(1-\rho)^2 \beta_1^2}\bar{F}_{\beta}(t)+o(\bar{F}_{\beta}(t))\quad\mbox{ as }t\to\infty.
\end{equation}

\begin{remark} \label{rem:4.2}
In Assumption A, we have assumed $a>2$ to ensure that the service time $T_{\beta}$ has a finite second moment $\beta_2$, which allows us to use
the results in \cite{Willekens1992}.
Here, we would like to mention that for the situation where $\beta_2=\infty$, the second order asymptotic properties of a random sum with $2RV$ summands have been studied partially in Leng and Hu (2014)~\cite{Leng-Hu:2014}, and their result is presented for certain cases with different combinations of parameter values $a$ and $h$.
Applying the result in \cite{Leng-Hu:2014}, one may proceed with a discussion on those cases with $1<a\le 2$ in a way similar to what follows, except that more complicated expressions would be involved in the derivations.
\end{remark}

Furthermore, in virtue of (\ref{6-3}) and (\ref{T-kappa-2nd-order}), one can immediately write out the second order asymptotic expansion for $P\{T_{\theta}>t\}$ as follows:
\begin{equation}\label{T-theta>x}
    P\{T_{\theta}>t\} = \frac {\lambda r_1} {(a-1)(1-\rho)} t^{-a+1}+\Delta_{T_{\theta}}(t),\quad\mbox{ as }t\to\infty,
\end{equation}
where
\begin{equation}\label{T-theta>x-2}
\Delta_{T_{\theta}}(t) = \left \{ \begin{array}{ll}
\displaystyle \left[\frac {\lambda^2\beta_2 r_1} {(1-\rho)^2 }+\frac {\lambda r_2} {a(1-\rho)} L_0(t)\right]t^{-a}+o(t^{-a}L_0(t)), & \mbox{for $h=1$}, \\
\displaystyle \frac {\lambda r_2} {(a+h-1)(1-\rho)} t^{-a-h+1}L_0(t) +o(t^{-a-h+1}L_0(t)), & \mbox{for $0<h<1$}, \\
\displaystyle \frac {\lambda^2\beta_2 r_1} {(1-\rho)^2 }t^{-a}+o(t^{-a}), & \mbox{for $h>1$}. \end{array} \right.
\end{equation}

The following two lemmas will be used to complete the two-term asymptotic expansion of $P\{T_{\theta}+T_{\beta}+T_{\tau}>t\}$ as $t\to\infty$.

\begin{lemma}[p.48 in \cite{Foss2011}]
\label{Lemma 3.4new}
Let $F$, $G_1$ and $G_2$ be distribution functions.
Suppose that $F\in\mathcal S$.
If $\bar{G}_i(t)/\bar{F}(t)\to c_i$ as $t\to\infty$ for some $c_i\ge 0, \; i=1,2$, then
$\overline{G_1*G}_2(t)/\bar{F}(t)\to c_1+c_2$ as $t\to\infty$,
where the symbol $G_1*G_2$ stands for the convolution of $G_1$ and $G_2$.
\end{lemma}

\begin{lemma}[Lemma~6.2 in \cite{Liu-Min-Zhao:2017}] \label{lemma-2nd-order-X}
    Let $X_1$ and $X_2$ be independent r.v.s with distribution functions $F_1$ and $F_2$, respectively.
Assume that $\overline{F}_1(t)\sim c_1 t^{-d+1}L(t)$ and $\overline{F}_2(t)\sim c_2 t^{-d}L(t)$, where $c_1>0$, $c_2>0$, $d>1$ and $L(t)$ is a slowly varying function at infinity. Then
\begin{eqnarray}\label{lemma-2nd-order}
P\{X_1+X_2>t\}=\overline{F}_1(t)+ (d-1)\mu_{F_2}\cdot t^{-1}\overline{F}_1(t) +\overline{F}_2(t)+o(\overline{F}_2(t))\quad\mbox{ as }t\to\infty,
\end{eqnarray}
where $\mu_{F_2}<\infty$ is the mean value of $X_2$.
\end{lemma}

It follows from (\ref{main}) that
$P\{T_{\tau}>t\} = (1-1/a)c_{\kappa}\psi r_1 t^{-a}+o(t^{-a})$ as $t\to\infty$ where $\psi$ and $c_{\kappa}$ are given in (\ref{psi}) and (\ref{c-K}), respectively. By Lemma \ref{Lemma 3.4new} and Assumption A, we have
\begin{equation}\label{T-beta+T-tau>x}
P\{T_{\beta}+T_{\tau}>t\}
= \left(r_1+\frac {\lambda r_1 } {a\mu (1-\rho)^2}\right)  t^{-a}+o(t^{-a}).
\end{equation}

In terms of Lemma~\ref{lemma-2nd-order-X}, (\ref{T-theta>x}), (\ref{T-theta>x-2}) and (\ref{T-beta+T-tau>x}), together with $E(T_{\beta}+T_{\tau})=\beta_1+\psi=\beta_1+\frac {\rho} {\mu(1-\rho)}$, we have the following result.
\begin{theorem} \label{the:tailT}
The second order asymptotic expansion for the tail probability $P\{T_{\theta}+T_{\beta}+T_{\tau}>t\}$ is given by:
\begin{equation}\label{P{T-theta+beta+tau}>x}
    P\{T_{\theta}+T_{\beta}+T_{\tau}>t\} = \frac {\lambda r_1} {(a-1)(1-\rho)} t^{-a+1}+\Delta_{T}(t),\quad\mbox{ as }t\to\infty,
\end{equation}
where
\begin{eqnarray}%
    \Delta_{T}(t) &=& \left \{ \begin{array}{ll}
\displaystyle \left[c_T+\frac {\lambda r_2} {a(1-\rho)} L_0(t)\right]t^{-a}+o(t^{-a}), & \mbox{for $h=1$}, \\
\displaystyle \frac {\lambda r_2} {(a+h-1)(1-\rho)} t^{-a-h+1}L_0(t) +o(t^{-a-h+1}), & \mbox{for $0<h<1$}, \\
\displaystyle c_T t^{-a}+o(t^{-a}), & \mbox{for $h>1$}, \end{array} \right.
\label{Delta-T-1}\\
c_T&=&\frac {r_1} {1-\rho}+\frac { \lambda r_1(\rho+1/a)} {\mu (1-\rho)^2}+\frac {\lambda^2\beta_2 r_1} {(1-\rho)^2 }.\label{Delta-T-2}
\end{eqnarray}

\end{theorem}

\section{Second order asymptotic expansion of tail probability for $L_{\mu}$}
\label{sec:5}

In this section, we will complete the proof of our main result (Theorem~\ref{the:main}), which is a refinement of the asymptotic equivalence (\ref{Lmu=Linfty}),
by introducing a second term in the asymptotic approximation to $P\{L_{\mu}>j\}$.
Towards this end, we now provide the following results, which will be used in the proof.
\begin{lemma}\label{lemma5-1}
Let $d>0$ be a constant. Then, as $x\to\infty$,
\begin{equation}
\frac {\Gamma(x-d)} {\Gamma(x)} =  x^{-d}+\frac {d(d+1)} {2} x^{-d-1}+O(x^{-d-2}).
\end{equation}
\end{lemma}
\begin{proof}
The second order asymptotic expansion of the gamma function is given by
\begin{equation}\label{Gammma/Gamma-1}
\Gamma(x) = \sqrt{2\pi}x^{x-1/2}e^{-x}
\left[1+\frac 1 {12x} + O\left(\frac 1 {x^2}\right)\right]\quad\mbox{ as }x\to\infty.
\end{equation}
Taking the logarithm of both sides, we get
\begin{equation}\label{Gammma/Gamma-2}
\ln\Gamma(x) = \left(x-1/2\right)\ln x -x +\ln \sqrt{2\pi} +\frac 1 {12x} +O\left(\frac 1 {x^2}\right).
\end{equation}
Noting that
\begin{eqnarray}\label{Gammma/Gamma-3}
(x-d-1/2)\ln (1-d/x)&=&(x-d-1/2)\left[-\frac d {x} - \frac {d^2} {2x^2}+O\left(\frac 1 {x^3}\right)\right]\nonumber\\
&=&-d +\frac {d(d+1)} {2x}+O\left(\frac 1 {x^2}\right),\nonumber
\end{eqnarray}
we can write
\begin{equation}\label{Gammma/Gamma-4}
\ln\Gamma(x-d) = (x-d-1/2)\ln x + \frac {d(d+1)} {2x} -x +\ln\sqrt{2\pi}+\frac 1 {12x} +O\left(\frac 1 {x^2}\right).
\end{equation}
By (\ref{Gammma/Gamma-2}) and (\ref{Gammma/Gamma-4}), we have $$\ln\Gamma(x-d)-\ln\Gamma(x)=-d\ln x + \frac {d(d+1)} {2x}+O(x^{-2}).$$ Therefore,
\begin{equation}
\frac {\Gamma(x-d)} {\Gamma(x)} = x^{-d}\exp\left\{\frac {d(d+1)} {2x}+O\left(\frac 1 {x^2}\right)\right\}=
x^{-d}\left [1 +\frac {d(d+1)} {2x}+O\left(\frac 1 {x^2}\right)\right ].
\end{equation}
\end{proof}
\begin{corollary}\label{corollary5-1}
Suppose that $\lambda>0$ and $d>0$. Then, as $j\to\infty$,
\begin{eqnarray}
\int_0^{\infty}\lambda e^{-\lambda t}\frac {(\lambda t)^{j+1}} {(j+1)!}\cdot t^{-d}dt
&=&\lambda^{d} j^{-d}+\frac 1 2 d(d-3) \lambda^{d} j^{-d-1} +O(j^{-d-2}).
\end{eqnarray}
\end{corollary}
\begin{proof}
Note that
\begin{eqnarray}\label{corollary5-1-proof-a}
\int_0^{\infty}\lambda e^{-\lambda t}\frac {(\lambda t)^{j+1}} {(j+1)!}\cdot t^{-d}dt &=&\lambda^{d}\cdot\frac {\Gamma(j+2-d)} {\Gamma(j+2)}.
\end{eqnarray}
By Lemma \ref{lemma5-1}, %$\Gamma(j+2-d)/\Gamma(j+2)=(j+2)^{-d}+\frac {d(d+1)} {2}(j+2)^{-d-1} +O(j^{-d-2})$
\begin{eqnarray}
\Gamma(j+2-d)/\Gamma(j+2)&=&(j+2)^{-d}+\frac 1 2 d(d+1) (j+2)^{-d-1} +O(j^{-d-2})\nonumber\\
&=&j^{-d}+\frac 1 2 d(d-3) j^{-d-1} +O(j^{-d-2}),\nonumber
\end{eqnarray}
which is substituted into (\ref{corollary5-1-proof-a}) to complete the proof.
\end{proof}

Suppose that $N_t$ is a Poisson process with rate $\lambda>0$, and $T>0$ is an
independent r.v. with tail probability $\bar{F}(t)=P\{T>t\}$. Then,
\begin{equation}
P\{N_T>j\} = \int_0^{\infty}P\{N_t\ge j+1\}dF(t)
=\int_0^{\infty}\lambda e^{-\lambda t}\frac {(\lambda t)^{j+1}} {(j+1)!}\bar{F}(t)dt.\label{P(N_T>j)}
\end{equation}
\begin{lemma}[Proposition~3.1 in  \cite{Asmussen-Klupperlberg-Sigman:1999}, or Theorem~3.1 in \cite{Foss-Korshunov:2000}]
\label{lemma-Asmussen-poisson}
If $\bar{F}(t)=P\{T>t\}$ is heavier than $e^{-\sqrt{t}}$ as $t\to\infty$, then $P(N_T>j)\sim P\{T>j/\lambda\}$ as $j\to\infty$.
\end{lemma}

Since $t^{-d}L(t)$ is heavier than $e^{-\sqrt{t}}$ as $t\to\infty$, we immediately have the following
corollary.
\begin{corollary}\label{corollary5-2}
Suppose that $\lambda>0$ and $d>0$. Let $L(t)$ be a slowly varying function at $\infty$. Then,
\begin{equation}
\int_0^{\infty}\lambda e^{-\lambda t}\frac {(\lambda t)^{j+1}} {(j+1)!}\cdot t^{-d}L(t)dt
\sim \lambda^{d} j^{-d}L(j)\qquad \mbox{ as }j\to\infty.
\end{equation}
\end{corollary}

By Lemma~\ref{lem:interpretations}, $L_{\mu}=N_{T_{\theta}+T_{\beta}+T_{\tau}}$.
Setting $\bar{F}(t)=P\{T_{\theta}+T_{\beta}+T_{\tau}>t\}$ in (\ref{P(N_T>j)}), we know
\begin{equation}
P\{L_{\mu}>j\}=P\{N_{T_{\theta}+T_{\beta}+T_{\tau}}>j\}
=\int_0^{\infty}\lambda e^{-\lambda t}\frac {(\lambda t)^{j+1}} {(j+1)!}\cdot P\{T_{\theta}+T_{\beta}+T_{\tau}>t\}dt.\label{P{N-theta+beta+tau}>j}
\end{equation}
Recall that the second order asymptotic expansion of $P\{T_{\theta}+T_{\beta}+T_{\tau}>t\}$ as $t\to\infty$. Substituting (\ref{P{T-theta+beta+tau}>x}) into (\ref{P{N-theta+beta+tau}>j}),
\begin{equation}
P\{L_{\mu}>j\}= \frac {\lambda r_1} {(a-1)(1-\rho)}\int_0^{\infty}\lambda e^{-\lambda t}\frac {(\lambda t)^{j+1}} {(j+1)!}\cdot t^{-a+1}dt
+\int_0^{\infty}\lambda e^{-\lambda t}\frac {(\lambda t)^{j+1}} {(j+1)!}\cdot \Delta_{T}(t)dt. \label{P{N-theta+beta+tau}>j-2}
\end{equation}
Substituting $\Delta_{T}(t)$, given in (\ref{Delta-T-1}) and (\ref{Delta-T-2}), into (\ref{P{N-theta+beta+tau}>j-2}) and
applying Corollary~\ref{corollary5-1} and Corollary~\ref{corollary5-2}, we complete
the proof of Theorem~\ref{the:main}.

\section{Concluding remarks and numerical results}
\label{sec:concluding}

In this paper, we considered the $M/G/1$ retrial queueing system, for which we obtained the second order asymptotic expansion for the tail probability of the total number $L_\mu$ of customers in the system. This result is a refinement of the first order approximation (the equivalence theorem (\ref{Lmu=Linfty}) for the retrial queue systems), and is also a refinement of a recent asymptotic result in \cite{Liu-Min-Zhao:2017}, which is the first order approximation for the difference between the tail probabilities of $L_\mu$ and $L_\infty$ (the total number of customers in the corresponding standard $M/G/1$ queueing system).
However, the first order asymptotic expansion for $P\{L_\mu > j\} - P\{L_\infty > j\}$ cannot directly act as the second term in the second order asymptotic expansion for $P\{ L_{\mu}>j \}$. This can be clarified as follows.

It was proved in \cite{Liu-Min-Zhao:2017} (Theorem~6.1) that, under the assumption $P\{T_{\beta}>t\} \sim t^{-a}L(t)$ as $t\to\infty$ where $a>1$,
\begin{description}
\item[(i)] $P\{L_{\mu}>j\}\sim P\{L_{\infty}>j\} \sim c_1 j^{-a+1}L(j)$;
\item[(ii)] $P\{L_{\mu}>j\}-P\{L_{\infty}>j\}\sim c_2 j^{-a}L(j)$,
\end{description}
where $c_1$ and $c_2$ are constants.
(i) and (ii) imply that
\begin{equation}
    P\{L_{\infty}>j\}= c_1\cdot j^{-a+1}L(j)+o(j^{-a+1}L(j)) \label{f2}
\end{equation}
and
\begin{equation}
    P\{L_{\mu}>j\}=P\{L_{\infty}>j\}+c_2\cdot j^{-a}L(j)+o(j^{-a}L(j)). \label{f1}
\end{equation}
Substituting (\ref{f2}) into  (\ref{f1}), we have
\begin{equation}
P\{L_{\mu}>j\}=c_1 j^{-a+1}L(j)+o(j^{-a+1}L(j))+ c_2 j^{-a}L(j)+o(j^{-a}L(j)). \label{f3}
\end{equation}
It is clear that compared to all other terms in (\ref{f3}), the last term $o(j^{-a}L(j))$ is a higher order infinitesimal function as $j \to \infty$. However, to specify the second term in the second order asymptotic expansion, a comparison of the orders between the second and the third terms in (\ref{f3}) is needed, which has not been addressed in Theorem~6.1 of \cite{Liu-Min-Zhao:2017}. This requires further specification on the second term of the asymptotic tail probability of the service time. Such a comparison has been made in this paper under Assumption~A, in which the
slowly varying function $L(\cdot)$ is further specified, allowing us to obtain the second term $o(j^{-a+1}L(j))$ in the second order asymptotic expansion of $L_\infty$.

Finally, it is interesting to see the improvement in the approximation when the second term in the asymptotic expansion is added to the first term approximation.
Let us set $r_1=r_2=1/2$ and $L_0(\cdot)\equiv 1$ in Assumption A. Thus, $\bar{F}_{\beta}(t)=t^{-a}\left(\frac 1 2+\frac 1 2 t^{-h}\right)$ with $t\ge 1$, where $a>2$ and $h>0$. The first two moments of service time are:
$\beta_1=\int_0^{\infty}\bar{F}_{\beta}(t)dt=1+\frac 1 2\left(\frac 1 {a-1}+\frac 1 {a+h-1}\right)$ and
$\beta_2=\int_0^{\infty}2t\bar{F}_{\beta}(t)dt=1+ \frac 1 {a-2}+\frac 1 {a+h-2}$. We use the notations Apprx1 and Apprx2 to represent the first and second order asymptotic approximations to $P\{L_{\mu}>j\}$, respectively. By Theorem~\ref{the:main},
$\mbox{Apprx1}= \frac {\lambda^a r_1} {(a-1)(1-\rho)} j^{-a+1}$ and $\mbox{Apprx2}=\mbox{Apprx1}+\Delta_L^{app}(j)$.
where $\Delta_L^{app}(j)$ is the approximated value of $\Delta_L(j)$ by neglecting the higher order infinitesimal in (\ref{the:main-2}). We use the notation Imprv to represent the relative improvement made by adopting the seond order asymptotic approximation, that is, $\mbox{Imprv}=\frac {\Delta_L(j)} {\mbox{Apprx2}}\times 100\%$. To be specific, we let $\mu=0.2$,
$a=2.5$,  $h= 2$,  and $\lambda=\rho/\beta_1$ with $\rho\in\{0.4, 0.6, 0.8\}$.
In the following table, we report the first and second order asymptotic approximations to $P\{L_{\mu}>j\}$, from which we can see a significant improvement.  Similar improvements can be seen for other parameter values.

%----Table-333333333333333333333333333333333333
\begin{table}[!hbp]
\begin{center}
\begin{tabular}{|c|c|r|r|r|}
\hline
\multirow{2}{*}{$\rho$} &\multirow{2}{*}{$j$}
&\multicolumn{3}{|c|}{$P\{L_{\mu}>j\}$}\\
\cline{3-5}
  &  &Apprx1 &Apprx1 &Imprv(\%) \\
\hline
\multirow{5}{*}{0.4}
 &10 	&0.00067146   &0.00092049       &37.088  \\
 &20 	&0.00023740   &0.00028142       &18.544  \\
 &30 	&0.00012922   &0.00014520       &12.363  \\
 &50 	&6.0057e-05   &6.4512e-05       &7.4176  \\
 &90 	&2.4869e-05   &2.5894e-05       &4.1209  \\
\hline
\multirow{5}{*}{0.6}
 &20 	&0.00098129   &0.00147690       &50.512  \\
 &40 	&0.00034694   &0.00043456       &25.256  \\
 &60 	&0.00018885   &0.00022065       &16.837  \\
 &80 	&0.00012266   &0.00013815       &12.628  \\
 &100 	&8.7769e-05   &9.6636e-05       &10.102  \\
\hline
\multirow{5}{*}{0.8}
 &20    &0.00402880  &0.01052500       &161.26  \\
 &60    &0.00077534  &0.00119210       &53.752  \\
 &100   &0.00036034  &0.00047656       &32.251  \\
 &160   &0.00017805  &0.00021394       &20.157  \\
 &200   &0.00012740  &0.00014795       &16.126  \\
\hline
\end{tabular}
\caption{The approximations to $P\{L_{\mu}>j\}$ for $a=2.5$ and $h=2$.}
\end{center}
\end{table}

\section*{Acknowledgments}
The authors would like to thank three anonymous reviewers and the handling editor for their valuable comments and suggestions, which improved the quality of this paper significantly. This work was supported in part by the National Natural Science Foundation of China (Grant No. 71571002),
the Natural Science Foundation of the Anhui Higher Education Institutions of China (No. KJ2017A340),
the Research Project of Anhui Jianzhu University, and a Discovery Grant from the Natural Sciences and Engineering Research Council of Canada (NSERC).

\end{document}